\documentclass[oneside,reqno,11pt]{amsart}
\usepackage[english]{babel}
\usepackage{indentfirst}
\usepackage[T1]{fontenc}
\usepackage[utf8]{inputenc}
\usepackage{paralist}
\usepackage{amssymb}
\usepackage{amsthm}
\usepackage{amsmath}
\usepackage{amscd}
\usepackage{graphicx}
\usepackage[colorlinks=true]{hyperref}
\usepackage[margin=1.2in]{geometry}
\usepackage[OT2,OT1]{fontenc}
\newcommand\cyr
{
\renewcommand\rmdefault{wncyr}
\renewcommand\sfdefault{wncyss}
\renewcommand\encodingdefault{OT2}
\normalfont
\selectfont
}
\DeclareTextFontCommand{\textcyr}{\cyr}

\makeatletter
\def\@settitle
{
\begin{center}
\baselineskip14\p@\relax
\LARGE\@title
\end{center}
}
\makeatother
\theoremstyle{plain} 
\newtheorem{thm}{Theorem}
\newtheorem{cor}[thm]{Corollary} 
\newtheorem{lem}[thm]{Lemma}
\newtheorem{prop}[thm]{Proposition}

\theoremstyle{definition}
\newtheorem{oss}[thm]{Remark}
\newtheorem{example}{Example}[section]
\newtheorem*{assump*}{Assumptions}
\newcommand{\R}{\ensuremath{\mathbb{R}}} 
\newcommand{\Rn}{\ensuremath{\mathbb{R}^n}} 
\newcommand{\eps}{\ensuremath{\varepsilon}} 

 \newcommand{\eqlab}[1]{\begin{equation}  \begin{aligned}#1 \end{aligned}\end{equation}} 
\newcommand{\bgs}[1]{\begin{equation*} \begin{aligned}#1\end{aligned}\end{equation*}} 
  \newcommand{\sys}[2][]{\begin{equation*}#1  \left\{\begin{aligned}#2\end{aligned}\right.\end{equation*}}
  
\newcommand{\E}{\ensuremath{\mathcal{E}}}

\newcommand{\C}{\ensuremath{\mathcal{C}}}
\newcommand{\K}{\ensuremath{\mathcal{K}}}

\newcommand{\X}{\ensuremath{\mathcal{X}}}
\newcommand{\B}{\ensuremath{\mathcal W^{s,p}_{R,\varphi} }}

\numberwithin{equation}{section} 

\title[A symmetry result in $\R^2$ for a general  type of nonlocal energy]{A symmetry result in $\R^2$ for global minimizers of a general  type of nonlocal energy}

\author[C. Bucur]{Claudia Bucur}

\address[C. Bucur]{
	Istituto Nazionale di Alta Matematica (INdAM)}
\email{claudia.bucur@aol.com}
\thanks{The author is member
	of {\em Gruppo Nazionale per l'Analisi Ma\-te\-ma\-ti\-ca, la Probabilit\`a e le loro Applicazioni} (GNAMPA) 
	of the {\em Istituto Nazionale di Alta Matematica} (INdAM)}
\thanks{I sincerely thank Enrico Valdinoci and Luca Lombardini for their very useful suggestions.}

\begin{document}
\maketitle

 \begin{abstract}In this paper, we are interested in a general type of nonlocal energy, defined on a ball $B_R\subset \Rn$ for some $R>0$ as
 \[ \E (u, B_R)= \iint_{\R^{2n}\setminus (\C B_R)^2} F( u(x)-u(y),x-y)\, dx \, dy+\int_{B_R} W(u)\, dx.\]  We prove that  in $\R^2$, under suitable assumptions on the functions $F$ and $W$, 
 bounded continuous global energy minimizers are one-dimensional. This proves a De Giorgi conjecture for minimizers in dimension two, for a general type of nonlocal energy. 
	\end{abstract}
\section{Introduction}
In this paper we deal with  
a general type of nonlocal energy. Let
\[ F\colon \R \times \Rn \setminus \{0\} \to [0,+\infty), \qquad W: [-1,1] \to [0,+\infty) \]  
be two functions and let $R>0$. Denoting as usual
\[ B_R =  \big\{ x\in \Rn \; \big| \; |x|<R\big \}, \qquad \C B_R = \Rn \setminus B_R,\]
we consider for any $|u|\leq 1,$
	\eqlab{\label{energy} 
		\E(u,B_R):= \K_R(u)  +\int_{B_R} W(u)\, dx,  
		}
with
	\eqlab{\label{kinn} 
	 	\K_R(u):= \iint_{\R^{2n}\setminus (\C B_R)^2} F(   u(x)-u(y),x-y)\, dx \, dy .
	} 
Under suitable assumptions on $F$ and $W,$ we prove that for $n=2$  continuous functions $u \colon \R^n\to [-1,1]$, 
minimizing the energy $\E(\cdot,B_R)$ for any $R>0$, 
 are one-dimensional. We say that $u$ is one-dimensional if every level set of $u$ is a hyperplane, or in other terms, if there exists $u_0 \colon \R \to [-1,1] $ such that 
 \[ u(x)= u_0(x\cdot  \omega) \qquad \mbox{ for some } \; \; \omega \in \partial B_1.\]

This type of energy naturally arises in a phase transition problem, which leads to  
 the well-known stationary Allen-Cahn equation 
	\eqlab{ \label{acclass}
		(-\Delta) u= u-u^3 \quad \mbox{ in } \, \Rn.
		}
 
The Italian mathematician Ennio De Giorgi conjectured in 1978 that any smooth, bounded solution of this  equation which is monotone in one direction is one-dimensional, at least if $n\leq 8$. The interested reader can check \cite{SavinSurvey} for a very nice survey on phase transitions, minimal surfaces, the Bernstein problem,  
since the connection between these problems is the reason why the dimension eight comes into play. For a further very nice reference, see \cite{CintiSurvey}.   \\ 
This De Giorgi conjecture has received much attention in the last decades, and has been completely settled for $n\leq 3$, see~\cite{AC00, BCN97,GGUI}. The case  $4\leq n \leq 8$ with the additional assumption that 	
\begin{equation} \label{limdgs}
 \lim_{x_n\to \pm \infty} u(x',x_n)=\pm 1, \quad \mbox{for any} \quad x'\in \R^{n-1}
\end{equation}
was proved in \cite{flatty}. On the other hand, an example showing that the De Giorgi conjecture does not hold in higher dimensions (i.e. for $n\geq 9$) can be found in \cite{PKW08}.
 
A model that accounts for long range interactions is given by the nonlocal, fractional counterpart of the Allen-Cahn equation 
	\eqlab{ \label{brr}
		 (-\Delta)^s u= u-u^3 \quad \mbox{ in } \, \Rn,
	}
	with $s\in (0,1)$. The operator $(-\Delta)^s$ denotes the fractional Laplacian 
	 defined as
	\bgs{\label{frlap}
		 (-\Delta)^s u (x)=C_{n,s}\int_{\Rn} \frac{2u(x)- u(x+y)-u(x-y)}{|y|^{n+2s}} \, dy, \quad \mbox { with }  \; C_{n,s}>0.
		 } 
		 

An analogue of the De Giorgi conjecture for any smooth, bounded, monotone solution of the fractional Allen-Cahn equation has first been proved for $n=2, s={1}/{2}$ is \cite{CM05}. In the case $n=2$, for any $s \in (0,1)$, the result is proved in \cite{CS15,SV09}. When $n=3$, the papers \cite{CC10, CC14} contain the proof for $s\in \big[{1}/{2},1\big)$, \cite{dipierro2017three,improvement} for  $s\in \big(0,{1}/{2}\big)$, and \cite{monoton} for a general $s\in (0,1)$.
For $n=4$ and $s={1}/2$ a proof of the conjecture is given in \cite{figalli}. On the other hand, for $4\leq n\leq 8$ and $s \in\big[{1}/{2},1\big)$ the conjecture is proved with the additional assumption \eqref{limdgs} in \cite{savin12,rigidity}. For $s$ in this range, only a counterexample for $n=9$ is missing to complete the picture.

\smallskip

One way to tackle the De Giorgi conjecture is to study global minimizers of the Ginzburg-Landau energy functional and understand whether they are one-dimensional. For the Allen-Cahn equation \eqref{acclass}, the related energy in some ball $B_R\subset\Rn$ is given by
	\bgs{ \label{classen} 
	\E(u,B_R) = \int_{B_R} \frac{1}2|\nabla u|^2 + W(u) \, dx,
		}
with $W$ being the  double-well potential
	\bgs{\label{DF-WELL}
		W(u)=\displaystyle \frac{(u^2-1)^2}{4}.
		}  
Actually, the potential term $W$ can denote any function with a double-well structure, that is
	\eqlab{
	\label{www1}
	 &W \colon[-1,1] \to [0,+\infty), \quad W\in C^2([-1,1]), \quad  W>0 \; \; \mbox{ in }\;  (-1,1) , \\ 
	 &W(\pm 1) = W'(\pm 1) =0,  \quad W''(\pm 1) >0.
	} 		
A local minimizer $u$ of the energy $\E(\cdot, B_R)$ is such that
$
 \E(u, B_R) \leq \E(v,B_R)$ for any $v=u$ on $\partial B_R.
 $ 
A global minimizer is a local minimizer for any $R>0$. 
It turns out that global minimizers of the Ginzburg-Landau energy (with $W$ as in \eqref{www1}) are one-dimensional for $n\leq 7$, see \cite{flatty} or \cite[Theorem 10.1]{SavinSurvey}. In fact, Savin proves the conjecture for global minimizers, and uses the additional assumption \eqref{limdgs} to go from global minimizers to solutions.  

 The nonlocal energy related to problem \eqref{brr} is 
		\eqlab{
		\label{frlapvv} 
			\E^s(u,B_R) = {\frac{1}2} \iint_{\R^{2n} \setminus (\C B_R)^2}\frac{|u(x)-u(y)|^2}{|x-y|^{n+2s}}\, dx \, dy + \int_{B_R}W(u) \, dx ,
	} 
	with $W$ satisfying \eqref{www1}.
	Here, a local minimizer $u$ of the energy $\E^s(\cdot, B_R)$ is such that
\[
 \E^s(u, B_R) \leq \E^s(v,B_R) \qquad \mbox{ for any }\;\; v=u \quad \mbox{ in } \; \; \C B_R,
 \] 
 and a global minimizer is a local minimizer in any ball. That nonlocal minimizers are one-dimensional  is proved for $n\leq 7$ and $s\in [1/2,1)$ in \cite[Theorems 1.2]{savin12,rigidity}, and the conjecture for solutions is settled (as in the classical case) by using the additional assumption \eqref{limdgs}. 
 
 In other references \cite{improvement,figalli,monoton}, the authors prove the conjecture with different techniques (for critical points of the energy, or for stable solutions). In \cite{CC10, CC14,CS15,CM05,dipierro2017three,SV09}, again various techniques are employed, but all rely on the use of the harmonic extension for the fractional Laplacian. However, there is not an ``extension procedure'' for general nonlocal operators, hence such methods are specific to the fractional Laplacian case. 
 On the other hand, in \cite[Theorem 4.2.1]{nonlocal}, the present author with Valdinoci carry out the proof of the conjecture for minimizers for $n=2$ in the nonlocal setting, thus without the harmonic extension. This allows to develop the technique therein introduced, and take the nonlocal energy in \eqref{frlapvv} to a much more general form. 
 
 We thus prove  the conjecture for global minimizers for $n=2$, for the general nonlocal energy given in \eqref{energy}. As a matter of fact, the results here introduced find as an immediate application the study of the energy related to \eqref{frlapvv}. Furthermore, the result applies also to more engaging equations, involving for instance the fractional $p$-Laplacian, or the mean curvature equation (as we see in Section \ref{examples}).

{We mention for the Allen-Cahn equation with general kernels  the papers \cite{cozzpass} and \cite{ros2015entire}. While reviewing this paper, we learned about the result reached in \cite{ros2015entire}. There, the one-dimensional property of stable solutions is proved in $\R^2$ for the operator
\[\mathcal L u(x)= P.V. \int_{\Rn} (u(x)-u(x+y)) K(y) \, dy
\]
under some assumptions on the kernel $K$,  and by using a Liouville theorem approach.} 

 \medskip
 
We organize the rest of the paper as follows. Section \ref{fff} contains the main result and the assumptions on the function $F$. {In Section \ref{missing} we deal with the existence of minimizers of the nonlocal general energy \eqref{energy} in an suitable functional setting. We discuss also some form of a strong comparison principle (i.e. if two ordered minimizers coincide on a small ball, then they coincide in the whole space)}. In Section \ref{prelim} we introduce some energy estimates, which will contribute to the proof of the main result (Theorem \ref{Theorem}) in Section \ref{thmm}.  We give two examples of functions $F$ that satisfy our assumptions in Section \ref{examples}. {As a matter of fact, in this last section we obtain that continuous bounded minimizers of the energy related to the fractional $p$-Laplacian and to the fractional mean curvature are one-dimensional in $\R^2$. In other words, we prove the De Giorgi conjecture for minimizers in dimension two (also) for the fractional $p$-Laplacian and the mean curvature.}

\section{Main result and assumptions on $F$}\label{fff}

We fix some $s\in (0,1)$ and $p\in [1,+\infty)$. 
We consider $F:\R\times \Rn\setminus \{0\} \to [0,+\infty)$ and denote by $t\in \R$, $x=(x_1,\dots,x_n)\in \Rn\setminus\{0\}$ its variables. 
 \begin{assump*}\label{fuffa} In this paper, $F$ satisfies the following on its domain of definition.
 \begin{itemize}

\item symmetry
		\eqlab{ 
			\label{sym} 
				F (t,x) =F(-t,x)=F(-t,-x) ,
			}
\item monotonicity in $t$
		\eqlab{\label{mon}
				  F(t_1, x)\leq F(t_2,x) \;\, \mbox{ for any } 
				  \; |t_1|\leq| t_2|, 
		  }
\item monotonicity in $x$
		\eqlab{
			\label{mon2} 
				 F(t, x_1)\leq F(t,x_2)\;\, \mbox{ for any }
				 \; |x_1|\geq | x_2|, 
			}
\item {scaling in $x$
		\eqlab{
			\label{homo} 
			 F\left(t,\alpha x\right)\leq \alpha^{-n-sp-1}  
			 F(t,x)\; \, \mbox{ for any } \alpha \in (0,1], 
	 }}
\item integrability: there exist $c_*,c^*>0$ such that
	{
		\eqlab{  \label{integr}  
	&c_* \left(\frac{|t|^p}{|x|^{n+sp} } - \frac{1}{|x|^{n+sp-p} }\right)
		\leq  
		F(t,x) 	
				 \leq c^* \frac{|t|^p}{|x|^{n+sp}},
				 }
		 }
\item  		smoothness in $x$ 
		\eqlab{
			\label{C2}     
				 &F(t, \cdot)\in C^2\left(\Rn \setminus \{0\}\right) 
				,
				 }
\item growth of the partial derivative in $x$: there exists $c_1>0$ such that
		\eqlab{ 
			\label{partial1} 
		  		\Big| \partial _{x_i} F (t,x)\Big|  \leq c_1 
		  		\frac{F(t,x)}{|x|} \;\,\, \mbox{ for any } i=1,2,\dots,n,
		}
\item growth of the second order partial derivative in $x$: there exists $c_2>0$ such that
		\eqlab{\label{partial2}   
			   \Big|   \partial^2_{x_i} F(t,x)\Big|  \leq 
		 	   c_2 \frac{F(t,x)}{|x|^2} \;\,  
			    \mbox{ for any } i=1,2,\dots,n, 
		    }
  \item  		smoothness in $t$
		\eqlab{
			\label{C2t} F(\cdot,x)\in C^1\left(\R \right) \;
				 \mbox{ for a.e. } x \in \Rn,
				} 
\item {growth of the derivative in $t$: there exists $c_3>0$ such that
		\eqlab{ 
			\label{1derivt} 
		  		\Big| \partial _{t} F (t,x)\Big|  \leq c_3 
		  		\frac{|t|^{p-1}}{|x|^{n+sp}},
		}}
		
		\item {strict monotonicity of $\partial_t F(t,x)$: 
			\eqlab{
				 \label{uconvex}
				\partial_t F(T,x) > \partial_t F(\tau,x),  \quad  \; &\mbox{ whenever }T> \tau,
	\mbox{ for any } x \in \Rn\setminus \{0\}.
				}}
\end{itemize}

 \end{assump*}

\noindent Let 
  \eqlab{
	\label{www}
	 &W \colon[-1,1] \to [0,+\infty), \quad W\in C^1([-1,1]), \quad  
	 W(\pm 1) = W'(\pm 1) =0.
	} 	 
  When $W$ satisfies \eqref{www}, we say that $u\colon\Rn \to [-1,1]$ is a minimizer for $\E(\cdot,B_R)$ given in \eqref{energy} if $\E(u, B_R)<\infty$ and if it minimizes $\E(\cdot,B_R)$ among all admissible competitors, i.e. 
\[ \E(u,B_R) \leq \E(v,B_R) \qquad \mbox{ for any }\; v \; \mbox{ such that } \quad |v|\leq 1 \quad \mbox{ and } \quad v=u \quad \mbox{ in } \; \; \C B_R.
 \] 
We say that $u\colon\Rn \to [-1,1]$ is a global minimizer for $\E$ if it is a minimizer for $\E(\cdot, B_R)$ for any $R>0$. 
\smallskip 
The main result of the paper is the following.
\begin{thm}\label{Theorem}
Let $u\colon \Rn \to [-1,1]$ be a continuous global minimizer of the energy \eqref{energy}. Then under the assumptions \eqref{sym} to 
\eqref{www}, $u$ is one-dimensional.
\end{thm}

\medskip

Notice that the assumptions on $F$ give a generalization of the energy  in \eqref{frlapvv}, related to the fractional Laplacian. { Moreover, we prove in the last Section \ref{examples}, they are all natural conditions when we consider a nonlocal energy like the one related to the fractional $p$-Laplacian or the fractional mean curvature operator.}


\section{{Existence and comparison of minimizers} }

\label{missing} 
Let $p\in[1,+\infty), s\in (0,1), R>0$ and let $ u\colon \Rn \to \R$ be a measurable function. We consider $F\colon \R \times \Rn\setminus \{0\}\to [0,+\infty)$ to be such that it satisfies at least \eqref{sym} and \eqref{integr} (other assumptions will be mentioned, when needed). Furthermore, let
\eqlab{
	\label{wloc}
	W \colon \R \to [0,+\infty), \qquad W\in L^\infty(\R)\cap C^1(\R)
	} 
	and let $\E(\cdot, B_R)$ be defined by the formula \eqref{energy}.
\\We begin by describing the functional framework for the existence (for further reference, check \cite{hitch}). Let
	\[
	W^{s,p}(\Omega) := \big\{ u\in L^p(\Omega) \; \big| \; [u]_{W^{s,p}(\Omega)} <\infty \big\}
	\]
where 
\[ 		
		[u]_{W^{s,p}(\Omega) }= \left(\int_\Omega \int_\Omega \frac{|u(x)-u(y)|^p}{|x-y|^{n+sp} }\, dx \, dy\right)^{\frac1p}
				\]
			is the Gagliardo semi-norm. Also, we denote 
				\[ 
				\|u\|_{W^{s,p}(\Omega)} :=\left(\|u\|^p_{L^p(\Omega)} +[u]^p_{W^{s,p}(\Omega) }\right)^{\frac1p}.\]
			  We define 
 	\bgs{ \mathcal X_R:= \big\{ \varphi \colon \Rn \to \R \; \big| \;  
\varphi \in L^\infty(\Rn)
 \cap W^{s,p}(B_{2R}) \big\}
  }
and denote
\[ [u]_{R,\varphi}:=\left(\int_{B_R} \left( \int_{B_{2R}\setminus B_R} \frac{|u(x)-\varphi(y)|^p}{|x-y|^{n+sp} }\, dy\right) \, dx  \right)^{\frac1p}.\] 
For $\varphi\in \X_R$, let
		 \eqlab{\label{thenorm}
		\mathcal {W}^{s,p}_{R,\varphi}:= \big\{u \colon \Rn \to \R \;\big| \;  u \in W^{s,p}(B_R),
		\,[u]_{R,\varphi}<\infty  
		\mbox{ and }u=\varphi \mbox{ on } \C B_R 
		\big\}.
		}
For $u\in \B$, when we say that $u$ is a minimizer for $\E(\cdot, B_R)$ it is implied that $u$ is a minimizer with respect to the fixed exterior data $\varphi$.
For the sake of precision, we recall that a measurable function $u\colon\Rn \to \R$ is a minimizer for $\E$ in $B_R$ if $\E(u, B_R)<\infty$ and   
\[ \E(u,B_R) \leq \E(v,B_R) \qquad \mbox{ for any }\;\; v=u \quad \mbox{ in } \; \; \C B_R.
 \] 
For any two sets $A, B\subset \Rn$, we define
	\[
		 u(A,B):= \int_{ A}\int_B F( u(x)-u(y),x-y)\, dx \, dy,
		 \]
		 and recall from \eqref{kinn} that
		 \[ \K_R(u)= u(B_R,B_R)+ u(B_R,\C B_R)+ u(\C B_R,B_R)
.\]
We have the next useful result.
\begin{prop}\label{finiteenergy}
If $\varphi\in \X_R$ and $u\in \B$, then there exists a positive constant $C$ depending on $n,s,p,R, \|W\|_{L^\infty(\R)}, \|\varphi\|_{L^\infty(\Rn)} $   such that
\[ \E(u,B_R)\leq  C \left( \|u\|^p_{W^{s,p}(B_{R})} 
			 +[u]^p_{R,\varphi}+  1 \right).\]
Moreover, it holds that
	\eqlab{\label{krusym}
		\K_R(u) = u(B_R,B_R) + 2u(B_R,\C B_R). 
		}
\end{prop}
\begin{proof}
By the right hand side of \eqref{integr} we have that 
	\bgs{\label{fuck1}
		& u(B_R,B_R) + 2u(B_R,B_{2R} \setminus B_R)\\
					\leq &\; c^* \int_{B_R}\int_{B_R}  \frac{|u(x)-u(y)|^p}{|x-y|^{n+sp}}  \, dx \, dy + 2 c^*\int_{B_R}\int_{B_{2R} \setminus B_R} \frac{|u(x)-\varphi(y)|^p}{|x-y|^{n+sp}}  \, dx \, dy
			\\		 
			   \leq &\;  c^*\left( [u]^p_{W^{s,p}(B_{R})} 
			 +2[u]^p_{R,\varphi}\right).
				}
		When $x\in B_R, y\in \C B_{2R}$, we have that $|x-y|\geq |y|/2$, hence
	\bgs{\label{fuck2}
		 u(B_R,\C B_{2R}) \leq &\; 
		c^* \int_{B_R} \int_{\C B_{2R}} \frac{|u(x)-\varphi(y)|^p}{|x-y|^{n+sp}}  \, dx \, dy
		 \\ 
		 \leq &\; 2^{p-1}  c^*  \left( \int_{B_R} |u(x)|^p \int_{\C B_{2R} } \frac{ dx \, dy}{|x-y|^{n+sp}} +  
		 \int_{B_R}  \int_{\C B_{2R} } \frac{ |\varphi(y) |^p}{|x-y|^{n+sp}}\,dx \, dy\right)
		 \\
		 \leq &\; 
		 	2^{p-1+n+sp}c^* \left(  \|u\|^p_{L^p(B_R)} + \|\varphi\|^p_{L^{\infty}(\C B_{2R})}|B_R|\right)  \int_{\C B_{2R}} |y|^{-n-sp}\, dy 
		 	\\
		 	\leq &\; C_{n,s,p,R}\left(  \|u\|^p_{L^p(B_R)} + \|\varphi\|^p_{L^{\infty}(\C B_{2R})}\right). }
		 	Therefore we obtain
		 	\eqlab{ \label{krrr}
		 	\K_R(u)\leq C_{n,s,p,R}\left( \|u\|^p_{W^{s,p}(B_R)} 
			 +[u]^p_{R,\varphi}+ \|\varphi\|^p_{L^{\infty}(\C B_{2R})}\right).
			 }
It is enough then to notice that
	\[ 
	\int_{B_R} W(u) \, dx \leq C_{n,R}\|W\|_{L^\infty(\R)}
	\]
	to conclude the first statement of the Proposition. \\
	On the other hand, changing variables, using Fubini to change the order of integration and applying \eqref{sym}, we obtain
	\bgs{ 
	&u(\C B_R, B_R) 
	\\
		=&\; \int_{\C B_R} \left( \int_{B_R} F(u(x)-u(y),x-y) \, dy \right)\, dx 
		= \int_{\C B_R} \left( \int_{B_R} F(u(y)-u(x), y-x) \, dx\right)  \, dy
		\\
		=&\;  \int_{B_R} \left(\int_{\C  B_R} F(u(x)-u(y), x-y) \, dy \right)\, dx
		= u(B_R,\C B_R),
	} 
	from which \eqref{krusym} immediately follows.
	\end{proof}

We give in the next proposition some a priori properties of the minimizers of the energy. 
 \begin{prop}\label{apriori} If $\varphi \in \X_R$ and $u$ is a minimizer of  $\E(\cdot, B_R)$ with $u=\varphi$ in $\C B_R$, then 
\begin{enumerate}
\item 
there exists $C=C_{n,s,p,R}>0$ such that
	\bgs{\label{rmk2}
		\E(u,B_R)\leq 
		 	C\left(\|\varphi\|^p_{W^{s,p}(B_{2R})} +\|\varphi\|^p_{L^{\infty}(\C B_{2R})} + \|W\|_{L^\infty(\R) }	\right),	}
	\item $u\in \B$
	 and furthermore there exists $c=c_{n,s,p,R}>0$ such that
	\eqlab{\label{bondonlp}
			 \|u\|_{L^p(B_R) } \leq c
			 (1+[u]^p_{W^{s,p}(B_R)}).
	 }
\end{enumerate}
\end{prop}
\begin{proof}
We can use $\varphi$ 
as a competitor for $u$. Using \eqref{krrr} for $\varphi$ (notice that $\varphi\in \B$), we obtain  
	\bgs{ \label{rm45} \K_R(\varphi) 
			\leq &\;C_{n,s,p,R}\left(\|\varphi\|^p_{W^{s,p}(B_{2R})} + \|\varphi\|^p_{L^\infty(\C B_{2R})}\right).
	}
	Given the minimality of $u$, we get that 
	\bgs{
			\E(u,B_R)\leq \E(\varphi,B_R)\leq 
			 C_{n,s,p,R}\left(\|\varphi\|^p_{W^{s,p}(B_{2R})} 
			 +\|\varphi\|^p_{L^{\infty}(\C B_{2R})} 
			 + \|W\|_{L^{\infty}(\R) }	\right) 	.
			}
			This proves point (1) of the proposition.
			By a change of variables, we obtain the bound
\eqlab{\label{buc}
		\int_{B_R} \int_{B_{2R}}  \frac{1}{|x-y|^{n+sp-p} }dx \, dy\leq |B_R| \int_{B_{3R}}  \frac{1}{|z|^{n+sp-p} }dz  =C(n,s,p,R).
		}
			According to the left hand side of \eqref{integr},  we have
	\bgs{ \label{nnnu1}
	u(B_R,B_R) \geq &\,{c_*} \left( \int_{B_R} \int_{B_R} \frac{|u(x)-u(y)|^p}{|x-y|^{n+sp} }dx \, dy - \int_{B_R} \int_{B_R}  \frac{1}{|x-y|^{n+sp-p} }dx \, dy \right) 
		 \\
		 =&\,{c_*} \left( [u]^p_{W^{s,p}(B_R)} - C_{n,s,p,R}\right)
	.}
	In the same way, we get that
	\[ u(B_R,B_{2R}\setminus B_R) \geq c_*\left( [u]^p_{R,\varphi} - C_{n,s,p,R}\right).\]
			Since $\E(u,B_R)$ is bounded, it holds that 
\eqlab{
	\label{this} [ u ]^p_{ W^{s,p}(B_R)} + [u]^p_{R,\varphi}\leq \E(u,B_R)+ C_{n,s,p,R} <C({n,s,p,R,\|\varphi\|_{L^\infty(\Rn)}, \|W\|_{L^\infty(\R)}}).
	} 
	Using Proposition \ref{pony} we have that
	\bgs{ 
	\|u\|^p_{L^p(B_R)} \leq &\; 
 C_{n,p,s,R} \left( [u]_{R,\varphi}^p  + \|\varphi\|^p_{L^p(B_{2R}\setminus B_R)} \right). 
			}
			This  implies that $u\in L^p(B_R)$, hence by \eqref{this}, we get that $u\in \B$. The bound on the $L^p$ norm \eqref{bondonlp} follows from Proposition \ref{poincy}. 
\end{proof}

\begin{oss} Let us note that there are some cases in which  the request
\[ [u]_{R,\varphi}<\infty\]
can be avoided. For $sp<1$, we can take
\bgs{ \mathcal X_R= \big\{ \varphi \colon \C B_R \to \R \; \big| \;  
\varphi \in
L^\infty(\C B_R)  \big\}.
  }
  In this case, we define 
 \bgs{
		\B:= \big\{u \colon \Rn \to \R \; \big| \; u_{\,\big|{B_R}} \in W^{s,p}(B_R)  \mbox{ and }u=\varphi \mbox{ on } \C B_R 
		\big\}.
		}
	Indeed, for $sp<1$, one can use the fractional Hardy inequality, thanks to \cite[Theorem D.1.4, Corollary D.1.5]{tesiluca} and get that
 	\bgs{
 		 \int_{B_R}\int_{B_{2R}\setminus B_R} \frac{|u(x)|^p}{|x-y|^{n+sp}}\, dx \, dy\leq &\,
 	 		\int_{B_R} \int_{\C B_{d_R(x)}(x)} \frac{|u(x)|^p}{|x-y|^{n+sp}}\, dx \, dy		
 	 		\\	
 	 		\leq &\, \int_{B_R} \frac{|u(x)|^p}{d_R(x)^{sp}}\, dx \leq  
 	 		C(n,s,p,R) \|u\|^p_{W^{s,p}(B_R)},
		}
		where $d_R(x)=dist(x,\partial B_R)$. 
		\\Just as a remark, the fractional Hardy inequality holds also for $sp>1$, see \cite[Theorem 1.1, (17)]{DydaHardy} for any $u\in C_c(B_R)$). Nevertheless, in this case one looks for minimizers in $W^{s,p}_0(B_R)$, a space which is too restrictive for our purposes. 
		
		\noindent Furthermore (check \cite[Lemma 4.5.10]{tesiluca}, or the forthcoming paper \cite{cl}) 
		\[
			 \int_{B_R}\int_{B_{2R}\setminus B_R} \frac{|\varphi(y)|^p}{|x-y|^{n+sp}}\, dx \, dy\leq \, \|\varphi\|_{L^\infty(\C B_R)} \mbox{Per}_{sp}(B_R) <\infty.
		\]
		This follows since the $sp$-perimeter  is finite for sets with Lipschitz boundary (see \cite{nms}).
		  Then
 	\bgs{ {[}u{]}^{p}_{R,\varphi}=&\;
 			\int_{B_R}\int_{B_{2R}\setminus B_R} \frac{|u(x)-\varphi(y)|^p}{|x-y|^{n+sp}} \,dx\,dy 
 			 					\leq  C(n,s,p,R)\left(|\varphi\|_{L^\infty(\C B_R)}
 					+  \|u\|^p_{W^{s,p}(B_R)}\right).		
 			}
 			We also notice that, in order to obtain the estimates  in Proposition \ref{apriori}, one can consider
 			\sys[\tilde \varphi=]{&\varphi, \quad \mbox{ in } \C B_R\\
 						&0, \quad \mbox{ in  }B_R
 						.}
 			\end{oss}

				
 We prove now the existence of minimizers of the energy.
\begin{thm} [Existence] \label{existence} Let $F$ be lower semi-continuous in the first variable, and such that it satisfies  \eqref{sym},  \eqref{integr}, and let $W$ be such that it satisfies  \eqref{wloc}. If $\varphi \in \X_R$,  
there exists a minimizer $u\in \B$ of $\E(\cdot, B_R)$.
\end{thm}
\begin{proof}
	 Since $F, W \geq 0$, we have that $\E(v,B_R)\geq 0$ for any $v\in \B$. Then there exists $\{u_k\}\in \B$ a minimizing sequence, i.e.
	\[
		\liminf_{k\to\infty} \E(u_k,B_R) = \inf \big\{ \E(v,B_R) \;\big| \; v\in \B\big\}.
		\]
	There is $\bar k >0$ such that for all $k\geq \bar k$ there exists $M>0$ such that
	\[
		\E(u_k,B_R) \leq M,
		\]
	so in particular by \eqref{this} we have that 
	\[
		[u_k]_{W^{s,p}(B_R)} \leq C_1, \qquad [u_k]_{R,\varphi}\leq C_2,
		\]
		with $C_1,C_2>0$ depending on $n,s,p,R,M$.
	Also, by  \eqref{bondonlp}, we have that
	\[
		\|u_k\|_{L^p(B_R)}<C(n,s,p,R) \left( [u]_{W^{s,p}(B_R)}+1\right),
		\] 
		therefore 
		for all $k\geq \bar k$ 
		there is $\tilde M>0$ such that
	\[
		\| u_k\|_{W^{s,p}(B_R)}<\tilde M.
		\]
	By compactness (see e.g. Theorem 7.1 in \cite{hitch}), there exists a subsequence, which we still call $\{u_k\}$, such that 
	\bgs{
		\|u_k-u\|_{L^p(B_R)}\longrightarrow 0, \quad \mbox{and }
		 \quad u_k \longrightarrow u \; \; \mbox{ a.e. in } \Rn
	}
		for some $u\in W^{s,p}(B_R)$. Also, $u\in \B$, by Fatou and the uniform bound on $[u_k]_{R,\varphi}$.  
		Using Fatou's Theorem, the lower semi-continuity of $F$ in the first variable and  \eqref{wloc} we have that
	\bgs{
		\inf \big\{ &\E(v,B_R) \,\big| \, v\in \B\big\} 
		\\
		=&\;\liminf_{k\to \infty} \bigg(  \iint_{\R^{2n}\setminus(\C B_R)^2} F(u_k(x)-u_k(y),x-y)\, dx \, dy  
		+ \int_{B_R} W(u_k)\, dx\bigg)
		\\
		\geq& \; \iint_{\R^{2n}\setminus(\C B_R)^2}\liminf_{k\to \infty}F(u_k(x)-u_k(y),x-y)\, dx \, dy 
		 + \int_{B_R} W(u)\, dx
		\\
		\geq  & \; \iint_{\R^{2n}\setminus (\C B_R)^2}F(u(x)-u(y),x-y)\, dx \, dy + \int_{B_R} W(u)\, dx
		\\
		=& \;\E(u,B_R).
		}
				Hence $u$ is a minimizer and this concludes the proof of the theorem.				
\end{proof}
We make now an observation on the Euler-Lagrange equation related to the energy $\E$. 

\begin{prop}\label{eulerlagr} 
Let $F$ satisfy \eqref{sym}, \eqref{integr}, \eqref{C2t}, \eqref{1derivt} and $W$ satisfy \eqref{wloc}. 
If $\varphi \in \X_R $ and $u\in \B$, then
	\bgs{\label{eq0}
		\frac{d}{d\eps} &\E(u+\eps\phi,B_R) \Big|_{\eps=0}\\
		=&\; \iint_{\R^{2n}\setminus (\C B_R)^2} \partial_t F(u(x)-u(y),x-y) (\phi(x)-\phi(y)) \, dx  \, dy + \int_{B_R} W'(u(x))\phi(x)\,dx
		}
		for any $\phi \in C^\infty_c(B_R)$.
\end{prop}
\begin{proof} We give a sketch of the proof.
First of all, notice that if $\phi\in C_c^\infty(B_R)$, for any $\eps>0$  we have that $u+\eps \phi \in \B$.  
By Proposition \ref{finiteenergy} it follows that both $\E(u,B_R)$ and $\E(u+\eps \phi,B_R)$ are finite. \\
Since $F(\cdot, x)\in C^1(\R)$ 
by the mean value theorem there is $\tau_\eps(x,y)$ satisfying $|\tau_\eps(x,y)|\leq \eps$ such that
	\bgs{
		 &\;\frac{ F\left(u(x)-u(y)+\eps(\phi(x)-\phi(y)),x-y\right)-F\left(u(x)-u(y),x-y\right)}\eps
		 \\
		 =&\; \partial_t F\left(u(x)-u(y)+ \tau_\eps (\phi(x)-\phi(y)),x-y\right) \left(\phi(x)-\phi(y)\right).
		 }
The assumption \eqref{1derivt}
		and the H{\"o}lder inequality 
lead to
	\bgs{
		\big|  \partial_t F\left(u(x)-u(y)+ \tau_\eps (\phi(x)-\phi(y)),x-y\right) \left(\phi(x)-\phi(y)\right)\big| \leq F(x,y),
	}
for some $F \in L^1(\R^{2n}\setminus (\C B_{R})^2)$. It is enough to use the Dominated Convergence Theorem to conclude the proposition.
\end{proof}

Furthermore, we prove 
some form of a strong comparison principle for minimizers. 

\begin{thm}\label{maxp} Let $F$ satisfy  \eqref{sym}, \eqref{integr}, \eqref{C2t}, \eqref{1derivt} and \eqref{uconvex} and let $W$ satisfy \eqref{wloc}. If $\varphi_1, \varphi_2 \in \X_R$  and $u_1 \in \mathcal W^{s,p}_{R,\varphi_1}$,  $u_2 \in \mathcal W^{s,p}_{R,\varphi_2}$ are two minimizers of $\E(\cdot, B_R)$, such that  
	\bgs{ & u_1,u_2 \in L^\infty(B_R), 
	\\ 
	&u_1\geq u_2\quad \mbox{ in } \R^n  
	\\
	& u_1=u_2 \quad \mbox{ in } B_\delta(\overline x) \subset \subset B_R
	}
	  for some  $\delta>0, \, \overline x\in B_R$, then $u_1=u_2$ almost everywhere in $\Rn$.
\end{thm}

\begin{proof}
			 	  According to Proposition \ref{eulerlagr} we have that 
		\bgs{
		\iint_{\R^{2n}\setminus (\C B_R)^2} &\Big( \partial_t F(u_2(x)-u_2(y),x-y) - \partial_t F(u_1(x)-u_1(y),x-y)  \Big)(\phi(x)-\phi(y)) \, dx \, dy 
		\\
		&+\int_{B_R} \left(W'(u_2(x))-W'(u_1(x))\right) \phi(x) \, dx=0
			}
			for any $\phi \in C^\infty_c(B_R)$.
In particular this equality holds for any
		\[ \phi\in  C^\infty_c(B_\frac{\delta}2(\bar x)), \qquad \phi \geq 0. 
		\] 
			Since $\phi(x)=0$ on $\C B_\frac{\delta}2(\bar x)$ and $u_1(x)=u_2(x)$ in $B_\delta(\bar x)$, contributions come only from interactions  between $B_\frac{\delta}2(\bar x)$ and $\C B_\delta(\bar x)$. So, using also \eqref{sym}, we are left with
			\bgs{\label{brru}
			& 0=
			 \int_{B_\frac{\delta}2(\bar x) } \left( \int_{ \C B_\delta(\bar x) } \Big( \partial_t F(u_1(x)-u_2(y),x-y) - \partial_t F(u_1(x)-u_1(y),x-y) \Big)\, dy \right) \phi(x) \, dx 
			\\
			& + 
		\int_{ \C B_\delta(\bar x) }	\left( \int_{B_\frac{\delta}2(\bar x) }   \Big( \partial_t F(u_2(x)-u_1(y),x-y)- \partial_t F(u_1(x)-u_1(y),x-y) \Big) (-\phi(y)) \, dy \right)\, dx 
		\\
		 &\;=\, 2 	\int_{B_\frac{\delta}2(\bar x) } \left( \int_{ \C B_\delta(\bar x) } \Big( \partial_t F(u_1(x)-u_2(y),x-y) - \partial_t F(u_1(x)-u_1(y),x-y) \Big)\, dy \right) \phi(x) \, dx .
			}
 Let 
		\[ 
		A_{\delta}:=\big\{y\in \C B_{\delta} (\bar x)\; \big| \; u_1(y)> u_2(y)\big\}
		\] 
		and we argue by contradiction, supposing that 
		\eqlab{ \label{ad} |A_\delta|\neq 0.}
		 When $y \in \C A_\delta$, by hypothesis $u_1(y)=u_2(y)$, hence
		 \bgs{\label{brru} 
		 0=\int_{B_\frac{\delta}2(\bar x) } \left( \int_{ A_\delta } \Big( \partial_t F(u_1(x)-u_2(y),x-y) - \partial_t F(u_1(x)-u_1(y),x-y) \Big)\, dy \right) \phi(x) \, dx.
		}
		  Denoting for $x \in B_\frac{\delta}2(\bar x),y \in A_\delta$,
\[ h(x,y):= \partial_t F(u_1(x)-u_2(y),x-y) - \partial_t F(u_1(x)-u_1(y),x-y) ,\]
by \eqref{uconvex} 
we have that on $A_\delta$
			 \eqlab{\label{brruf}
				h(x,y) >0.
				}
Defining
		$g:B_\frac{\delta}2(\bar x)\to \R_+$  as 
		\[ g(x):= \int_{A_\delta} h(x,y)\, dy\]
		we get that for any $\phi\in C^\infty_c(B_\frac{\delta}2(\bar x),[0,+\infty))$ 
		\[
		0= \int_{B_\frac{\delta} 2(\bar x)} g(x) \phi(x) \, dx.
		\]
		It follows that
			\[ g(x)=0 \quad \mbox{ for almost any } x\in B_\frac{\delta}2 (\bar x),\] which by \eqref{ad} and \eqref{brruf} gives a contradiction.  
It follows that that $|A_\delta|=0$, hence $u_1 =u_2$ almost anywhere in $\C B_{\delta}(\bar x) $.\\
We conclude by noticing that, by \eqref{1derivt}, $g$ is well defined. Indeed
		\bgs{
		&\; \int_{A_\delta} \left| \partial_t F(u_1(x)-u_2(y),x-y)\right| \, dy \leq c_3 \int_{A_\delta} \frac{|u_1(x)-u_2(y)|^{p-1} }{|x-y|^{n+sp}}\, dy 
		\\
	\leq &\; 2^{p-2}c_3
	\bigg( \|u_1\|^{p-1}_{L^\infty(B_R)} \int_{\C B_\delta(\bar x)} |x-y|^{-n-sp}\, dy
	+\|u_2\|^{p-1}_{L^\infty(B_R)} \int_{B_R \setminus B_\delta(\bar x)}|x-y|^{-n-sp}\, dy 
	  \\
	 &\;  +  \|\varphi_2\|^{p-1}_{L^\infty(\C B_R)} \int_{\C B_R} |x-y|^{-n-sp}\, dy\bigg).
		} 
	We have that $|y-x|\geq |y-\bar x|-|x-\bar x|\geq |y-\bar x|/2$, hence
	\[ 
		\int_{A_\delta} \left| \partial_t F(u_1(x)-u_2(y),x-y)\right| \, dy\leq C_{n,s,p,\delta}  \left( \|u_1\|^{p-1}_{L^\infty(B_R)}+\|u_2\|^{p-1}_{L^\infty(B_R)}+\|\varphi_2\|^{p-1}_{L^\infty(\C B_R)} \right),
	\] 
	and this concludes the proof.	
\end{proof}


\section{Preliminary energy estimates}\label{prelim}
The preliminary results in this Section hold in any dimension, however the main result works with our techniques only in dimension two. In fact, this depends on a Taylor expansion of order two, that we do in the next Lemma.

\begin{lem}\label{first} Let $F$ satisfy \eqref{sym}
to  \eqref{partial2}, $W$ satisfy \eqref{wloc} and let
 $ \varphi \in C_c^{\infty}(B_1)$.  Also, for any $R>1$ and $y\in \Rn$, let
\[ \Psi_{R,\pm}(y):= y\pm \varphi\left(\frac{y}R\right) e_1 \quad \mbox{ and } \quad 
 u_{R,\pm}(x)= u(\Psi^{-1}_{R,\pm}(x)).\] Then for large $R$ the maps $ \Psi_{R,\pm}$ are diffeomorphisms on $\Rn$ and
\[ \E(u_{R,+},B_R) +\E(u_{R,-},B_R) -2\E(u,B_R) \leq \frac{C}{R^2} \E(u,B_R) .\]
\end{lem}

\begin{proof}

From here on, we denote for simplicity
\[ u=u(y), \; \bar u= u(\bar y), \; \varphi= \varphi\left( \frac{y}R \right), \;   \bar \varphi= \varphi\left(\frac{\bar y}R\right).\] 
Notice that
\eqlab{\label{normaphi} |\varphi  - \bar \varphi | \leq \frac{ \|\varphi\|_{C^1(\Rn)}}R |y-\bar y|}
and that for any $\delta\in[-1,1]$ 
  \eqlab{
 	 \label{stimamod}
  			| y-\bar y+ \delta e_1(\varphi-\bar \varphi)  | \geq 
  				\left(1-\frac{2\|\varphi\|_{C^1(\Rn)}}{R}\right)^{\frac12} 
  				|y-\bar y|.
  				}
  Indeed
  \[ | y-\bar y+ \delta e_1 (\varphi-\bar \varphi)  | ^2 = |y-\bar y|^2 +  \delta^2 (\varphi -\bar \varphi)^2 + 2 \delta(y_1-\bar y_1)(\varphi-\bar \varphi)  \geq |y-\bar y|^2 - 2| \delta| |y_1 -\bar y_1| |\varphi-\bar \varphi|.\] 
    Using \eqref{normaphi} we have that
  \[ 
  	| \delta| |y_1 -\bar y_1| |\varphi-\bar \varphi|
  	 \leq  \frac{\|\varphi\|_{C^1(\Rn)} }{R}   |y -\bar y|^2,
  \]
  hence \eqref{stimamod} is proved.

Now, checking Lemma 4.3 in \cite{nonlocal}, one sees that $\Psi_{R,\pm}$ are diffeomorphisms for large $R$, and that the change of variables
\bgs{\label{change1} x:= \Psi_{R,\pm}(y) , \qquad  \bar x= \Psi_{R,\pm}(\bar y) }
gives
\bgs{\label{dx1} dx= 1\pm \frac{1}R\partial_{x_1}\varphi \left(\frac{y}R\right) + \mathcal{O}\left(\frac{1}{R^2}\right)\, dy }
and
	\eqlab{ 
		\label{dxdy} dx\,d\bar x= 1\pm \frac{1}R\partial_{x_1}\varphi  \pm  \frac{1}R \partial_{x_1} \bar \varphi  
		+ \mathcal{O}\left(\frac{1}{R^2}\right)\, dy\, d\bar y.
	}
With this change of variables, we have that
	\bgs{
		\label{FRR} F(u_{R,\pm}(x)-u_{R,\pm}(\bar x), x-\bar x)&= F (u(\Psi^{-1}_{R,\pm}(x))-u(\Psi^{-1}_{R,\pm}(\bar x)), x-\bar x) \\ &=F\left(u(y)-u(\bar y), y-\bar y +e_1 \left( \pm \varphi\mp \bar \varphi\right) \right).
		}
		Notice that $\Psi_{R,\pm}^{-1}(B_R)= B_R$ and 	 $\Psi_{R,\pm}^{-1}(\C B_R)= \C B_R$. Changing variables we have that
	\eqlab{\label{psibr}
		&\iint_{\R^{2n}\setminus (\C B_R)^2}  F(u_{R,\pm}(x)-u_{R,\pm}(\bar x), x-\bar x) \, dx \, d\bar x  
		\\
		= & \iint_{\R^{2n}\setminus (\C B_R)^2} F(u-\bar u, y-\bar y \pm e_1(\varphi-\bar \varphi)) 
			\left(1\pm \frac{1}R\partial_{x_1}\varphi  \pm  \frac{1}R \partial_{x_1} \bar \varphi  
		+ \mathcal{O}\left(\frac{1}{R^2}\right)\right)\, dy\, d\bar y.
		}				
Thanks to \eqref{stimamod}, \eqref{mon2} and \eqref{homo}, for any $R$ large enough and any $\delta \in [-1,1]$ we have the estimate
	\eqlab{ \label{primader}
				F(u-\bar u, y-\bar y \pm \delta e_1(\varphi-\bar \varphi)) 
				\leq &\,
				F\left(u-\bar u,\left(1-\frac{2\|\varphi\|_{C^1(\Rn)} }{R}\right)^{\frac12} (y-\bar y)\right)
				\\ 
				\leq &\,
				\left(1-\frac{2\|\varphi\|_{C^1(\Rn)} }{R}\right)^{\frac{-n-sp-1}{2} } F(u-\bar u, y -\bar y).
				} 				
We define the function 
		\[ 
			g\colon\R\to \R_+, \quad g(h):=F(u-\bar u, y-\bar y + h e_1 (\varphi-\bar \varphi) )
	\] 
 and we have that
 \eqlab{\label{gzero}
 		g(0)= F(u-\bar u, y-\bar y) , \quad g(\pm 1)=F(u-\bar u, y-\bar y \pm e_1(\varphi- \bar \varphi) ).
 		}
 Also, we take the derivatives
	 \bgs{     			
 			 	& g'(h)= \partial_{x_1} F ( u-\bar u, y-\bar y+he_1 ( \varphi-\bar \varphi ) ) (\varphi-\bar \varphi ) ,
 			 	\\
 			 	 &  g''(h)= \partial^2_{x_1} F ( u-\bar u, y-\bar y+h e_1( \varphi-\bar \varphi ) ) (\varphi-\bar \varphi )^2.
 			 	} 
Using \eqref{partial1},  \eqref{normaphi}, \eqref{stimamod} and \eqref{primader}  we obtain
	\bgs{
			|g'(h)|\leq  &\, c_1 |F ( u-\bar u, y-\bar y+he_1 ( \varphi-\bar \varphi ) )| \frac{|\varphi -\bar \varphi|  }{|y-\bar y+he_1 ( \varphi-\bar \varphi )|}
			\\
			\leq &\; c_1  \left(1-\frac{2\|\varphi\|_{C^1(\Rn)} }{R}\right)^{\frac{-n-sp-2}{2}} \frac{\|\varphi\|_{C^1(\Rn)} } R
			F(u-\bar u, y-\bar y) ,
			}
			hence
			\eqlab{ \label{gprimo}
				 |g'(h)|\leq g(0)\mathcal O \left(\frac{1}{R}\right).
				 }
			In the same way,  using  \eqref{partial2}  we get that
			\eqlab{ \label{gsec}
				 |g''(h)|\leq g(0)\mathcal O \left(\frac1{R^2}\right).
				 }
By \eqref{C2} since $g\in C^2(\R)$ with a Taylor expansion we have  
		\[ 
		g(h)=g(0)+g'(\delta)h
			\]
 for some $\delta =  \,\delta(h) \in (0,h)$, hence
 \bgs{\label{g11} g(1)= g(0)+g'(\delta_+) ,\qquad g(-1)= g(0)-g'(\delta_-), \qquad \textcolor{black}{\mbox{ for some }}\delta_+\in(0,1), \delta_- \in (-1,0).}
 Moreover, there exists $\tilde \delta \in (\delta_-,\delta_+)$ such that
 \eqlab{
 		\label{secondg} g'(\delta_+) - g'(\delta_-) = g''(\tilde \delta) 				(\delta_+-\delta_-).
 	} 
So with this Taylor expansions and formula \eqref{psibr} we obtain
 \eqlab{ \label{krru11}
  \K_R(u_{R,+}) + \K_R(u_{R,-}) = &
 		\iint_{\R^{2n}\setminus (\C B_R)^2}  g(1)
			\left(1+ \frac{1}R\partial_{x_1}\varphi  +  \frac{1}R \partial_{x_1} \bar \varphi  
		+ \mathcal{O}\left(\frac{1}{R^2}\right)\right) 
		 \\
		&\quad +g(-1)
  \left(1- \frac{1}R\partial_{x_1}\varphi  -  \frac{1}R \partial_{x_1} \bar \varphi  
		+ \mathcal{O}\left(\frac{1}{R^2}\right)\right)\, dy\, d\bar y
		  \\ 
		=  &\; \iint_{\R^{2n}\setminus (\C B_R)^2} g(0) \left(2+\mathcal O\left(\frac{1}{R^2}\right) \right) \, dy\,d\bar y  
			\\
			& + \; \iint_{\R^{2n}\setminus (\C B_R)^2} \left(g'(\delta_+) -g'(\delta_-) \right)  \left(1+\mathcal O\left(\frac{1}{R^2}\right) \right) \, dy\,d\bar y
 \\
 		&+\;\iint_{\R^{2n}\setminus (\C B_R)^2} \frac{1}R   \left( g'(\delta_+) + g'(\delta_-)\right)  (\partial_{x_1} \varphi + \partial_{x_1} \bar \varphi)  \, dy\,d\bar y   
 \\
 = 	&\; \iint_{\R^{2n}\setminus (\C B_R)^2}g(0) \left(2+\mathcal O\left(\frac{1}{R^2}\right) \right)+T_1(y,\bar y)+T_2(y,\bar y) \, dy \, d\bar y.}
 In order to have an estimate on $T_1$, we use \eqref{secondg} and get that
 \bgs{T_1(y,\bar y) \leq  & \; \Big| g'(\delta_+) -g(\delta_-) \Big|  \left(1+\mathcal O\left(\frac{1}{R^2}\right) \right) \, dy\,d\bar y  \\  
    \leq&\;  \Big| g''(\tilde \delta) \Big|  \left(2+\mathcal O\left(\frac{1}{R^2}\right) \right) \, dy\,d\bar y
     ,} 
  where we have used that $\delta_+-\delta_-\leq2$. By \eqref{gsec} we obtain
  \[T_1(y,\bar y)\leq g(0)\mathcal O\left(\frac{1}{R^2}\right).\] 
   On the other hand
 \bgs{
 		 T_2(y,\bar y) \leq\frac{2\|\varphi\|_{C^1(\Rn)} } R 
		 \left(|g'(\delta_+)|+|g'(\delta_-)|\right) , 
		    }
		    which by \eqref{gprimo} leads to
		    \[ T_2(y,\bar y) \leq g(0)\mathcal O\left( \frac{1}{R^2} \right).\]
Therefore in \eqref{krru11} we have that
  \bgs{& 
  		\K_R(u_{R,+}) + \K_R(u_{R,-}) \leq \iint_{\R^{2n}\setminus (\C B_R)^2} g(0) \left(2+\mathcal O\left(\frac{1}{R^2}\right) \right) \, dy\,d\bar y  =  \K_R(u)  \left(2+\mathcal O\left(\frac{1}{R^2}\right) \right).}
  For the potential energy, the computation easily follows. It suffices to apply the change of variables \eqref{change1} and to recall that  $\Psi_{R,\pm}^{-1}(B_R)= B_R$. We get
 \bgs{ \int_{B_R} W(u_{R,+}(x)) \, dx &+ \int_{B_R}W(u_{R,-}(x) )\, dx\\
 		= &\;\int_{B_R}W(u (\Psi^{-1}_{R,+}(x)))\, dx +\int_{B_R}W(u (\Psi^{-1}_{R,-}(x))) \, dx\\
 		= &\; \int_{B_R}W(u (y)) \bigg(1+ \frac{1}R\partial_{x_1}\varphi \left(\frac{y}R\right)+ \mathcal{O}\left(\frac{1}{R^2}\right)\bigg)\, dy  \\
 		&\;+ \int_{B_R}W(u (y)) \bigg(1- \frac{1}R\partial_{x_1}\varphi \left(\frac{y}R\right) + \mathcal{O}\left(\frac{1}{R^2}\right)\bigg)\, dy\\
 		=&\; \bigg(2+\mathcal{O}\left(\frac{1}{R^2}\right)\bigg)  \int_{B_R}W(u (y)) \, dy.}  This concludes the proof of Lemma \ref{first}.
\end{proof}

We give now the following uniform bound on large balls of the energy of the minimizers. This result is an adaptation of Theorem 1.3 in \cite{densityEs} and it works in any dimension.

\begin{thm}\label{thmunif} Let $F$ satisfy \eqref{sym}, \eqref{mon} and \eqref{integr} and $W$ satisfy \eqref{www}.
If $u$ is a minimizer in $B_{R+2}$ for a large $R$, such that $|u|\leq 1$, then
\[ \E(u,B_R)\leq  \begin{cases}
						CR^{n-1} \quad &\mbox{if} \quad s\in \big(\frac{1}{p},1\big),\\
						CR^{n-1}\log R \quad &\mbox{if} \quad s=\frac{1}{p},\\
						CR^{n-sp} \quad &\mbox{if} \quad s\in \big(0,\frac{1}{p}\big),
					\end{cases}
	\]
	for some positive constant $C$ depending on $n, s$ and $W$.
\end{thm}

This type of energy estimates for the fractional Laplacian are proved in \cite[Theorem 1.3]{densityEs} (see also \cite[Theorem 4.1.2]{nonlocal}). We give a sketch of the proof following \cite{densityEs}, pointing out which assumptions on $F$ make the proof work in our case.

\begin{proof}[Proof of Theorem \ref{thmunif}]
As a first step, one introduces the auxiliary functions 
	\bgs{\psi(x):=-1+2 \min&\big\{ (|x|-R-1)_+,1\big\}, \quad v(x):=\min \big\{u(x), \psi(x)\big\}, \\& 
	 d(x):= \max\big\{ (R+1-|x|),1 \big\}.}
	 Notice that $|\psi|, |v| \leq 1$.  For $|x-y|\leq d(x)$ we have that 
	\begin{equation}\label{acpsi1} |\psi(x)-\psi(y)|\leq \frac{2 |x-y|}{d(x)}.\end{equation} Moreover,
		one obtains the estimate
	\begin{equation} \label{acdx1} \int_{B_{R+2}} d(x)^{-sp}\, dx \leq 
	\begin{cases} 
			 CR^{n-1}  \quad & \mbox{if}  \quad s\in \big(\frac{1}{p},1\big),\\
			 CR^{n-1}\log R   \quad &\mbox{if} \quad s=\frac{1}{p},\\
			 CR^{n-sp}  \quad & \mbox{if}  \quad s\in \big(0,\frac{1}{p}\big) .
		\end{cases}
	\end{equation}
	Also, by \eqref{integr} and \eqref{acpsi1} we get that
	\bgs{ &\int_{\Rn} F(\psi(x)-\psi(y),x-y) \, dy \leq c^*\int_{\Rn} \frac{|\psi(x)-\psi(y)|^p}{|x-y|^{n+sp}}\,dy  \\
	\leq &\; c^*\int_{|x-y|\leq d(x)  } \frac{|\psi(x)-\psi(y)|^p}{|x-y|^{n+sp}} \,dy +c^*\int_{|x-y|\geq  d(x)  } \frac{|\psi(x)-\psi(y)|^p}{|x-y|^{n+sp}}\,dy \\
	\leq &\;c^*\,  d(x)^{-p} \int_{|x-y|\leq d(x)}  |x-y|^{p-n-sp} \, dy  + c^*\,\int_{|x-y|\geq d(x)} |x-y|^{-n-sp }\, dy \leq c d(x)^{-sp}.}
It follows that
\bgs{ \E(\psi, B_{R+2}) \leq&\;  \int_{B_{R+2} } \left(\int_{\Rn}  F(\psi(x)-\psi(y),x-y) \, dy\right)\, dx + \int_{B_{R+2}} W(\psi)\, dx\\
\leq &\;c \int_{B_{R+2}} d(x)^{-sp} \, dx+ \int_{B_{R+2}} W(\psi)\, dx.} 
Moreover $W(-1)=0$ and $\psi=-1$ on $B_{R+1}$, so
\[ \int_{B_{R+2}	} W(\psi)\, dx = \int_{B_{R+2}\setminus B_{R+1} } W(\psi)\, dx\leq C R^{n-1}.\]
	With this, we obtain the bound
	\begin{equation} \label{psir2} \E(\psi, B_{R+2})  \leq 
		\begin{cases} 
			 CR^{n-1}  \quad & \mbox{if}  \quad s\in \big(\frac{1}{p},1\big),\\
			 CR^{n-1}\log R   \quad &\mbox{if} \quad s=\frac{1}{p},\\
			 CR^{n-sp}  \quad & \mbox{if}  \quad s\in \big(0,\frac{1}{p}\big) ,
		\end{cases}
	\end{equation}
	where $C=C(n,s,p)>0$. \\ Letting
	\[ A:=\{v=\psi\},\] 
	we notice that $B_{R+1}\subseteq A\subseteq B_{R+2}$ and that for $x\in A, y\in \C A$
	\[ |v(x)-v(y)|\leq \max\big\{|u(x)-u(y)|, |\psi(x)-\psi(y)|\big\}.\] 
	Then by \eqref{mon}
	we have that
	\[ F(v(x)-v(y),x-y)  \leq F(u(x)-u(y),x-y)+ F(\psi(x)-\psi(y),x-y)   .\]
Integrating on $A\times \C A$ we get that
	\[v(A,\C A)\leq u(A,\C A)+ \psi (A,\C A).\]
	We recall that $u$ is a minimizer in $B_{R+2}$, and $u=v$ outside of $B_{R+2}$ (and outside of $A$), so
	\bgs{ 0\leq &\; \E(v,B_{R+2})-\E(u,B_{R+2}) =\E(v,A)-\E(u,A) .}
	Since $v=\psi$ on $A$, it follows that
	\bgs{ u(A,A) + \int_A W(u)\, dx \leq \E(\psi,A)}
	and, given that $B_{R+1}\subseteq A\subseteq B_{R+2}$, 
		\bgs{u(B_{R+1},B_{R+1}) + \int_{B_{R+1}} W(u)\, dx \leq \E(\psi,B_{R+2}).} Also, one has that
	\bgs{ u(B_R,\C B_{R+1}) \leq &\;C \int_{B_{R+2} } d(x)^{-sp} \, dx 
	\leq  \begin{cases} 
			 CR^{n-1}  \quad & \mbox{if}  \quad s\in \big(\frac{1}{p},1\big),\\
			 CR^{n-1}\log R   \quad &\mbox{if} \quad s=\frac{1}{p},\\
			 CR^{n-sp}  \quad & \mbox{if}  \quad s\in \big(0,\frac{1}{p}\big).
		\end{cases}}
		Using this together with the estimate \eqref{psir2}, we obtain the claim of Theorem \ref{thmunif}.
	\end{proof}

We have the following very useful lemma.  
\begin{lem}\label{maxmin} Let $F$ be convex in the first variable and such that it satisfies \eqref{sym}, and  let $W$ be such that it satisfies \eqref{wloc}.
Let   $\Omega$ be a measurable set and $u,v\colon \R^n \to \R$ be two measurable functions. Let \[ m:=\min\{u,v\}, \qquad M:=\max\{u,v\},\]
then
\[ \E(m,\Omega) +\E(M,\Omega) \leq \E(u,\Omega)+\E(v,\Omega).\]
\end{lem}
We omit the proof of this known result, a complete proof is given e.g. in  \cite[Lemma 4.5.15]{tesiluca} (and the forthcoming paper \cite{cl}), while a general abstract version in \cite{sunra}.

\begin{prop}\label{blah}
\noindent If $W$ satisfies \eqref{www}, let
	\sys[ \tilde W:=]{ &W, & \mbox{ in } \; &[-1,1]\\
						&0, & \mbox{ in } \; & \R \setminus [-1,1]
						,} 
			\[	
			\E (u,B_R) := \K_R(u) + \int_{B_R} W (u) \, dx, \qquad\mbox{ and } \qquad  \tilde \E (u,B_R) := \K_R(u) + \int_{B_R} \tilde W (u) \, dx.
				\]
				For any measurable function $\tilde u\colon \Rn \to \R$, denote also
				\[
				  u := \max\left\{ \min\left\{ \tilde u,1\right\}, -1\right\}.
				\]
				
				a) It holds that
				$\E(u,B_R) = \tilde \E (u, B_R) \leq \tilde \E (\tilde u, B_R)$.
				
				\smallskip			
				\noindent Furthermore, if $\varphi \in \X_R$ is such that $|\varphi|\leq 1$ and 
				 			 
									b) if $u\in \B$ is a minimizer for $\E(\cdot, B_R)$, then $u$ is a minimizer for $\tilde \E (\cdot, B_R)$;

					 c) if $\tilde u\in \B$ is a minimizer for $\tilde \E(\cdot, B_R)$, then $u$ is a minimizer for $\E(\cdot, B_R)$.
		
				\end{prop}
			
\begin{proof}
Notice at first that $\tilde W$ satisfies \eqref{wloc}. Then by Lemma \ref{maxmin} and using the notations therein we have that
\[ \max \big\{ \tilde \E(m, B_R), \tilde \E(M,B_R)\big \}  \leq \tilde \E(u,B_R) + \tilde \E(v,B_R).\]
given that $F,W\geq0$ (hence $\tilde \E(w,B_R) \geq 0$ for any measurable function $w$). 
Moreover, since $W(\pm 1)=0$ we get that  
	\[ 
		\tilde \E(\pm 1, B_R)=0.
	\] 
	Therefore denoting 
	\bgs{
		& \bar u =\min \{ \tilde u,1 \}, \qquad 
		}
			we obtain
	\[
		\tilde \E(  u , B_R ) \leq \tilde \E( \bar u ,B_R) \leq\tilde \E (\tilde u, B_R),
		\]
		which  concludes a).\\		
		For b), let $\tilde w\colon \Rn \to \R$ be any competitor for $u$, hence such that $\tilde w=\varphi$ on $\C B_R$. Denoting
		\[
		 w := \max\left\{ \min\left\{ \tilde w,1\right\}, -1\right\},
		\]
		by the minimality of $u$ and according to a) we have that
		\[ 
			\tilde \E (u, B_R) =\E (u,B_R) \leq  \E(w, B_R) =  \tilde \E(w, B_R) \leq \tilde \E (\tilde w, B_R).
		\]
		On the other hand, let $w\colon \Rn \to [-1,1]$ be any admissible competitor for $u$, thus such  $w=\varphi$ on $\C B_R$. From the minimality of $\tilde u$ and  according to a) we have that
		\[ 
			\E(u,B_R) =\tilde \E ( u, B_R) \leq \tilde \E (\tilde u, B_R) \leq \tilde \E (w, B_R) =  \E ( w, B_R),
		\] 
		thus c) is proved.
			\end{proof}	
			

\section{Proof of the main result}\label{thmm}
The proof of the main result follows the step of \cite [Theorem 4.2.1]{nonlocal}. We underline  the main ideas from \cite{nonlocal}, and focus on the new computations  needed for the type of energy here introduced.

\begin{proof}[Proof of Theorem \ref{Theorem}]
We organize the proof in four steps.

\medskip
\noindent\textbf{Step 1. A geometrical consideration} \\In order to prove that $u$ is one dimensional, one has to prove that the level sets of $u$ are hyperplanes. 
It is thus enough to prove that $u$ is monotone in any direction. Using this, one has that the level sets are both convex and concave, thus flat.

\medskip
\noindent \textbf{Step 2. Energy estimates} \\
Let $R>8$ and $\varphi \in C_c^{\infty}(B_1)$ such that $\varphi=1$ in $B_{1/2}$, and let $e=(1,0)$. 
Witgh the notations of Lemma \ref{first} we notice that
 \begin{align}
		&u_{R,\pm} (y)= u(y) \; &\text{for} \; &y \in \C B_R \label{urpiur1}\\
		&u_{R,+} (y)= u(y-e) \; &\text{for} \; &y \in B_{R/2}\label{urpiur2} ,
		\end{align} 
		and 
		\[  |u_{R,\pm}|\leq 1.\] 		
		We use the notations of Proposition \ref{blah} for $\tilde W, \tilde \E$. Since $u$ is a minimizer for $\E$ in $B_R$ and $u_{R,-}$ is a competitor, thanks to Lemma \ref{first} (applied to $\tilde \E$) and to Proposition \ref{blah} a),  we have that
	\[ \E(u_{R,+},B_R) -\E(u,B_R) \leq \E(u_{R,+},B_R) +\E(u_{R,-},B_R)- 2 \E(u,B_R)\leq \frac{C}{R^2}  \E(u,B_R).\] From Theorem \ref{thmunif}  applied for $n=2$ it follows that
		\eqlab{\label{blaa2} \lim_{R\to \infty} \left(\E(u_{R,+},B_R)- \E(u,B_R) \right)=0.}
		{We remark that this is the crucial point where we require $n=2$. }
	
	\medskip
\noindent\textbf{Step 3. Monotonicity} \\
Suppose by contradiction that $u$ is not monotone in any direction. So, denoting $e=(1,0)$, up to translation and dilation, we suppose that
\[ u(0) >u(e), \qquad \mbox{ and } \quad u(0)>u(-e).\]
For $R$ large enough, we denote
	\[
	 v_R(x):=\min\{ u(x), u_{R,+}(x)\}, \qquad w_R(x):=\max\{ u(x), u_{R,+}(x)\}	
	\] 
and remark that $v_R, w_R$ are continuous and that $|v_R|,|w_R|\leq 1$.
		So $v_R=w_R=u$ on $\C B_R$ and since $u$ is a minimizer
\[ \E(w_R,B_R)\geq \E(u,B_R).\]
Moreover, by Lemma \ref{maxmin} (applied to $\tilde \E$) and  Proposition \ref{blah} a), we have that
\[ \E(v_R,B_R) + \E(w_R,B_R)\leq \E(u,B_R)+\E(u_{R,+},B_R),\]
therefore 
\eqlab{\label{blaa1} \E(v_R,B_R)\leq \E(u_{R,+},B_R).}
Since $u(0)= u_{R,+}(-e)$ and $u(-e)=u_{R,+}(0)$, {using the  continuity of the two functions $u$ and $u_{R,+}$}, we obtain that
\bgs{ &v_R<u \; \mbox{ in a neighborhood of } 0\\
 & v_R=u \mbox{ in a neighborhood of } -e.} This implies that $v_R$ is not identically nor $u$, nor $u_{R,+}$.
 
  We remark now that $v_R$ is not a minimizer of $ \E(\cdot, B_2)$. Indeed, $u$ is a global minimizer, hence $u,v_R\in W^{s,p}(B_R)$ for any $R$, so $u,v_R\in \X_2$. Moreover,
\bgs{
		& |u|,|v_R|\leq 1 &\mbox{ in } \; & \R^2 ,\\
		  &	v_R \leq u &\mbox{ in } \; & \R^2, 
		  \\
		  & v_R= u & \mbox{ in } \; & B_\delta(-e) \; \; \mbox{ for some } \delta>0. 
		  }
	If also $v_R$ is a minimizer for $\E(\cdot, B_2)$, then from Proposition \ref{blah} b), $u,v_R$ are minimizers for $\tilde \E (\cdot, B_2)$. Theorem \ref{maxp} (applied to $\tilde \E$) implies that $v_R=u$ on $B_2$, which gives a contradiction.
		  
 According to Theorem \ref{existence} and to  Proposition \ref{blah} c), there exists  $v_R^*$ a minimizer of $ \E(\cdot, B_2)$ such that $v_R^*=v_R$  on $\C B_2$ and $|v_R^*|\leq 1$.
Let
 \[ \delta_R:=\E(v_R,B_2)- \E(v_R^*,B_2)\geq 0.\] 
 We prove that there exists an universal constant $c>0$ such that $\lim_{R\to \infty} \delta_R\geq c.$ 
For this, we define as in Theorem 4.2.1 in \cite{nonlocal}
\bgs{&\; \tilde u(x)=u(x-e), \; \; m(x)=\min\{u(x),\tilde u(x)\}, \qquad  |\tilde u|, |m| \leq 1,
}
and observe that $m$ is not a minimizer for $ \E(\cdot, B_2)$. 
Indeed, $u,m\in W^{s,p}(B_R)$ for any $R$, so $u,m\in \X_2$. Moreover,
\bgs{
		& |u|,|m|\leq 1 &\mbox{ in } \; & \R^2 ,\\
		  &	m \leq u &\mbox{ in } \; & \R^2, 
		  }
		   and 
 \bgs{ 
 	&m<u \; \mbox{ in a neighborhood of } 0\\
 	& m=u \mbox{ in a neighborhood of } e.
 	} 
 Using also Proposition \ref{blah} b), if $m$ were a minimizer for $\E(\cdot, B_2)$, we would obtain a contradiction by the comparison principle in Theorem \ref{maxp}. According to Theorem \ref{existence} and Proposition \ref{blah} c), there exists $z$ a minimizer of $ \E(\cdot, B_2)$ such that $m=z$ in $\C B_2$ and $|z|\leq 1$. Furthermore, there exists some $c>0$ (independent of $R$) such that
 \bgs{\E(m,B_2) - \E(z,B_2)  :=c.}
Let 
 \bgs{
		&\; z_R(x):= \psi(x) z(x) +(1-\psi(x)) v_R(x),}
 with $\psi \in C^{\infty}_c(\Rn, [0,1])$ a cut-off function such that 
 \sys[\psi(x)=]{& 1, && x\in B_{R/4}\\
 				&0, && x\in \C B_{R/2},}
 				and notice that $|z_R|\leq 1$.
 	It holds that 
 	\bgs{ & m= v_R, \; \; z=z_R&&  \mbox{ in } B_2\\
 		& m= v_R= z=z_R=v_R*&&  \mbox{ in } B_{\frac{R}2}\setminus B_2\\
 		& v_R=z_R=v_R^*, \;\; m=z &&  \mbox{ in } B_R\setminus B_{\frac{R}2}\\
 		& m= z, \; \; u=v_R= v_R^*=z_R&&  \mbox{ in } \C B_{R}.}
 		With this in mind, we get that
	\eqlab{\label{ghf}
		 & c= \E(m,B_2) -\tilde \E(z,B_2) 
		 \\
		=&\; \E(m,B_2)-\E(v_R,B_2)+\delta_R + \tilde \E(v^*_R,B_2) -\tilde \E(z_R,B_2)+\tilde \E(z_R,B_2)-\tilde \E(z,B_2)\\
		\leq &\; \E(m,B_2)-\E(v_R,B_2)+ \E(z_R,B_2)- \E(z,B_2)+\delta_R ,
		}
		since $z_R$ is a competitor for $v_R^*$ in $B_2$. 
		Now, for $x\in B_2, y\in \C B_{R/2}$ we have $|x-y|\geq |y|/2$, hence
		\bgs{  
		&\E(m,B_2)-\E(v_R,B_2) 
		\\
		=  &\; 2\int_{B_2}\left( \int_{\C  B_{\frac{R}2} } F(m(x)-m(y),x-y) -F(m(x)-v_R(y), x-y)  dy \right)  dx
			\\
		\leq &\; 4 \int_{B_2} \int_{\C  B_{\frac{R}2} } F\left(2,\frac{y}2 \right) <C R^{-sp}   
		,}
given \eqref{integr} and \eqref{mon2}.
		 The same bound holds for $ \E(z_R,B_2)- \E(z,B_2) $.
It follows in \eqref{ghf} that
	\bgs{
		\label{nnn}& c \leq C R^{-sp} +\delta_R.
		}
Sending $R\to \infty$ we  get that
 \[ \lim_{R\to \infty} \delta_R \geq c>0.\]
 Finally, by minimality,
 \bgs{ \E(u,B_R) \leq  \E(v_R^*,B_R)=\E(v_R,B_R) -\delta_R \leq \E(u_{R,+},B_R) -\delta_R,}
 according to \eqref{blaa1}. Sending $R\to \infty$
 \[ c\leq  \lim_{R\to \infty} \left(\E(u_{R,+},B_R) - \E(u,B_R)\right) ,\]
 which contradicts \eqref{blaa2}. This concludes the proof of Theorem \ref{Theorem}.
\end{proof}


\section{Some examples}\label{examples}
{We give in this section some examples of problems to which our results can be applied. We give two example of functions $F$, that agree with the requirements \eqref{sym} to \eqref{uconvex}. As we see here, the context considered is general enough to be applied to the energy related to the fractional Laplacian, the fractional $p$-Laplacian and to the nonlocal mean curvature. }

We consider in this section $W$ as in \eqref{www}.

\begin{example} We consider
\eqlab{\label{plapl} F(t,x)=\frac{1}p \frac{|t|^p}{|x|^{n+sp}}}
for $p>1$ and $s\in (0,1)$. The nonlocal energy that we study is
	\eqlab{\label{plalplfrrr} 			
		\E(u,B_R)
		=\frac{1}p\iint_{\R^{2n}\setminus (\C B_R)^2} \frac{ |u(x)-u(y)|^p}{|x-y|^{n+sp}}\, dx \, dy
		 + \int_{B_R} W(u)\, dx.
		}
We note that the associated equation is given by
\[ (-\Delta)^s_p u +W'(u)=0,\] 
with $(-\Delta)^s_p$ being the fractional p-Laplacian, defined as 
\[(-\Delta)^s_p u (x)= P.V. \int_{\Rn} \frac{ |u(x)-u(y)|^{p-2} (u(x)-u(y))}{|x-y|^{n+sp}} \, dy .\] Notice that for $p=2$, we obtain the fractional Laplacian. The interested reader can see \cite{nonlocal,hitch,silvestre} and references therein for the fractional Laplacian, \cite{nhk,plapl,Korvenpa,korvenpaaobstacle} and references therein  for the fractional p-Laplacian, or more general fractional operators. 
 
 \medskip
 
It is not hard to verify that $F$ given in \eqref{plapl} satisfies 
\eqref{sym} to
\eqref{C2} and \eqref{C2t} to \eqref{uconvex}. Also, \eqref{partial1} and \eqref{partial2} follow after simple computations and hold for
\[ c_1= n+sp, \quad c_2=(n+sp)(n+sp+1).\]  

It follows that for $n=2$ bounded global  minimizers of \eqref{plalplfrrr} are one-dimensional, more precisely we have the following.
\begin{cor}
Let $u\colon \Rn \to [-1,1]$ be a continuous global minimizer of the energy \eqref{plalplfrrr}, with $W$ satisfying \eqref{www}. Then $u$ is one-dimensional.
\end{cor}
\noindent We remark that \cite[Corollary 3.3]{cozzidg}  gives that a global minimizer of \eqref{plalplfrrr} is continuous.

\end{example}

\bigskip

\begin{example}We consider the function related to the fractional mean curvature equation. Nonlocal minimal surfaces were introduced in \cite{nms} as boundaries of sets that minimize a nonlocal operator, namely the fractional perimeter. The first variation of the fractional perimeter operator is the nonlocal mean curvature, defined as the weighted average of the characteristic function, with respect to a singular kernel (the interested reader can check \cite{abat,bucval,dipiersurv} and other references therein).  For smooth hypersurfaces that are globally graphs, i.e. taking $\partial E$ as a graph in the $e_n$ direction defined by a smooth function $u$, the nonlocal mean curvature is given by
	\[ \mathcal{I}_s[E](x,x_n) =2 P.V.\int_{\R^{n}} \frac{dy}{|x-y|^{n+s}} \int_0^{\frac{u(x)-u(y)}{|x-y|}} \frac{d\rho}{(1+\rho^2)^{\frac{n+s+1}2}} \]
	for $s\in (0,1)$ and $(x,x_n)\in \partial E$, with $u(x)=x_n$ and taking $\nabla u(x)=0$.  See e.g. \cite{caffy},  the Appendix in \cite{bucval} for the proof of this formula. Also, the interested reader may see \cite[Chapter 4]{tesiluca} (and the forthcoming paper \cite{cl}) for the the mean curvature (and related nonparametric minimal surfaces).
	\\
	 We use the notations
\[ g(\rho) = \frac{1}{(1+\rho^2)^{\frac{n+s+1}2}}, \qquad G(\tau) =\int_0^{\tau} g(\rho) \, d\rho, \qquad \mathcal G(t) =\int_0^t G(\tau) \, d\tau.\] Notice that \[  G'(\tau) =g(\tau), \qquad \mathcal G'(t) =G(t) .\]
With this, the nonlocal energy to study is
	\eqlab{
		\label{ennms} \E(u,B_R) =\iint_{\R^{2n}\setminus (\C B_R)^2} \frac{dx \, dy}{|x-y|^{n+s-1} }  \mathcal G\left(\frac{u(x)-u(y)}{|x-y|}\right)+ \int_{B_R} W(u)\, dx,
	}
and  the relative equation is
	\[
	P.V. \int_{\Rn} \frac{dy}{|x-y|^{n+s} }  G\left( \frac{u(x)-u(y)}{|x-y|} \right) + W'(u)=0,  
	\] 
	see also Theorem 1.10 in \cite{dipierro2017decay} for other applications.
So, we let
	\eqlab{
		\label{nms} F(t,x)= \frac{1}{|x|^{n+s-1} }   \mathcal G\left(\frac{t}{|x|}\right)
	} 
%
and prove that the requirements  \eqref{sym} to \eqref{uconvex} hold for $s\in (0,1)$ and $p=1$.\\
\noindent We notice that $g,\mathcal G$ are even, $G$ is odd and
	\eqlab{ 
		\label{elem} 0< g(\rho) \leq 1, \qquad  |G(t)| < \int_{-\infty}^{\infty} g(\rho)\, d \rho < C, \qquad 0\leq  \mathcal G(t) \leq C| t|.
	}
Also, the following chain of inequalities holds
\eqlab{\label{simpleineq} a^2g(a)\leq   aG(a) \leq 2 \mathcal G(a), \quad \mbox{ for any } a\geq 0.}  
Indeed, since $g$ is decreasing we have that
\[ G(a) = \int_0^a g(\rho) \, d \rho \geq \int_0^a g(a) \, d \rho= a g(a).\] 
Also, denoting $f(a)=aG(a)-2\mathcal G(a)$ we have that
\[ f'(a)=  a g(a) -  G(a) \leq 0.\]
So $f$ is decreasing and $f(a)\leq f(0)=0$ for $a\geq 0$. This proves the inequalities in \eqref{simpleineq}.
\\
It is easy to check that the assumptions \eqref{sym}, \eqref{mon},  \eqref{mon2} and \eqref{C2} hold for $F$ as in \eqref{nms}. Also for some $\beta>1$, since $g(\beta\rho)\leq g(\rho)$, we have that
\[ \mathcal G (\beta t)= \beta^2 \int_{0}^t d\tau \int_0^\tau g(\beta \rho)\, d\rho \leq  \beta^2 \mathcal G ( t).\]
 Thus for some $\alpha \in (0,1)$
 \[   F\left(t,\alpha x \right)\leq  \alpha^{-n-s-1} \frac{1}{|x|^{n+s-1}} \mathcal G \left(\frac{t}{|x|}\right),\] and we get \eqref{homo}.
Using \eqref{elem} we obtain
\[ F(t,x) = \frac{1}{|x|^{n+s-1}} \mathcal G\left(\frac{t}{|x|} \right) \leq c \frac{|t| }{|x|^{n+s}},\]
that is the right hand side of \eqref{integr} for $p=1$. The left hand side follows using the bound in  \cite[Lemma 4.2.1]{tesiluca} (and the forthcoming paper \cite{cl}), that is
	\[
		\mathcal G (\tau) \geq c_* (|\tau|-1). 
	\]
By computing the derivative with respect to $x_1$ of $F$, we get that
\[ |\partial_{x_1} F(t,x)| \leq (n+s-1) \frac{F(t,x)}{|x|} + \frac{1}{|x|^{n+s} } \bigg| \frac{t}{|x|} G\left(\frac{t}{|x|}\right)\bigg|\leq c_1  \frac{F(t,x)}{|x|} \]
thanks to \eqref{simpleineq}. Moreover
\[ |\partial^2_{x_1} F(t,x)| \leq C_1 \frac{F(t,x)}{|x|^2} + \frac{1}{|x|^{n+s+1} } \bigg| \frac{t}{|x|} G\left(\frac{t}{|x|}\right)\bigg| + \frac{1}{|x|^{n+s+1} } \frac{t^2}{|x|^2} g\left(\frac{t}{|x|} \right)  \leq c_2  \frac{F(t,x)}{|x|^2} \]again by using \eqref{simpleineq}. 
So \eqref{partial1} and \eqref{partial2} are satisfied. 
We see also that 
	\[\left| \partial_t F(t,x)\right| \leq \frac{1}{|x|^{n+s}} \left| G\left(\frac{t}{|x|}\right)\right| \leq \frac{C}{|x|^{n+s}},\] where \eqref{simpleineq} and \eqref{elem} were used. Assumptions \eqref{C2t} and \eqref{1derivt} follow. 
Moreover, it is obvious from the definition of $\mathcal G$ that
\[ \partial_t^2 F(t,x)= \frac{1}{|x|^{n+s+1}} g\left(\frac{t}{|x|}\right) >0.\] 
From this \eqref{uconvex} is straightforward. 

  Theorem \ref{Theorem} then says that in $\R^2$, global minimizers of the energy \eqref{ennms} are one-dimensional. To our knowledge, this is a new result in the literature. The precise result goes as follows.
 \begin{cor}
Let $u\colon \Rn \to [-1,1]$ be a continuous global minimizer of the energy \eqref{ennms} with $W$ satisfying \eqref{www}. Then $u$ is one-dimensional.
\end{cor}

We remark that, up to our knowledge, the continuity of minimizers of the energy \eqref{ennms} is not known. Even for the classical case the problem is quite a delicate one, the interested reader can consult e.g. \cite{bombi,cafgar,trudy,wang}. 
\end{example}

\begin{appendix} 
\section{Some known results}

\begin{prop}\label{pony} Let $\Omega\subset \mathcal O \subset \Rn$ be bounded, open sets such that $|\mathcal O \setminus \Omega|>0$ and let $u\colon \Rn \to \R$ be a measurable function. Then
\bgs{ 
			\|u\|^p_{L^p(\Omega)} \leq &\; \frac{2^{p-1}}{|\mathcal O \setminus \Omega|} \left( d_{\mathcal O}^{n+sp} \int_{\Omega}\ \int_{\mathcal O \setminus \Omega}  \frac{|u(x)-u(y)|^p}{|x-y|^{n+sp}} \,dx \, dy+ |\Omega|\,  \|u\|^p_{L^p(\mathcal O \setminus \Omega)} \right),
			}
			with $d_{\mathcal O}=\mbox{diam}(\mathcal O)$.
\end{prop}
\begin{proof}We have that
\bgs{
		|u(x)|^p = &\; |u(x)-u(y)+u(y)|^p 
		\\
		= &\;   \frac{1}{|\mathcal O \setminus \Omega|}  \int_{\mathcal O \setminus \Omega}  |u(x)-u(y)+u(y)|^p  \, dy
		\\
		\leq&\; \frac{2^{p-1}}{|\mathcal O \setminus \Omega|} \left( \int_{\mathcal O \setminus \Omega} \frac{|u(x)-u(y)|^p}{|x-y|^{n+sp}} |x-y|^{n+sp} + |u(y)|^p\, dy\right)
		\\
		\leq &\;  \frac{2^{p-1}}{|\mathcal O \setminus \Omega|} \left( 
		d_{\mathcal O}^{n+sp}\int_{\mathcal O \setminus \Omega}  \frac{|u(x)-u(y)|^p}{|x-y|^{n+sp}} \, dy
		+\int_{\mathcal O \setminus \Omega} |u(y)|^p\, dy\right).
		}
		The conclusion follows by integrating on $\Omega$.
		\end{proof}
		
		We recall also a fractional Poincar\'e inequality (see \cite[Proposition 2.1]{poing} for the proof).
		\begin{prop}[A fractional Poincar\'e inequality] \label{poincy}
		Let $\Omega\subset  \Rn$ be bounded, open set and let $u\colon \Rn \to \R$ be in $L^1(\Omega)$. Then
\bgs{ 
		\|u-u_{\Omega}\|_{L^p(\Omega)} \leq \left(\frac{d_{\Omega}^{{n+sp}}}{|\Omega|}\right)^{\frac{1}p}  [u]_{W^{s,p}(\Omega)},
		}
		where 
		\[ u_\Omega=\frac{1}{|\Omega|} \int_{\Omega} u(x) \, dx \qquad \mbox{ and } \qquad  d_{\Omega}=\mbox{diam}(\Omega).\] 
		
\end{prop}

\end{appendix}
\bibliography{biblio}

\begin{thebibliography}{10}

\bibitem{abat}
Nicola Abatangelo and Enrico Valdinoci.
\newblock A notion of nonlocal curvature.
\newblock {\em Numer. Funct. Anal. Optim.}, 35(7-9):793--815, 2014.

\bibitem{AC00}
Luigi Ambrosio and Xavier Cabr{\'e}.
\newblock Entire solutions of semilinear elliptic equations in {$\bold R^3$}
  and a conjecture of {D}e {G}iorgi.
\newblock {\em J. Amer. Math. Soc.}, 13(4):725--739 (electronic), 2000.

\bibitem{BCN97}
Henri Berestycki, Luis Caffarelli, and Louis Nirenberg.
\newblock Further qualitative properties for elliptic equations in unbounded
  domains.
\newblock {\em Ann. Scuola Norm. Sup. Pisa Cl. Sci. (4)}, 25(1-2):69--94, 1997.
\newblock Dedicated to Ennio De Giorgi.

\bibitem{bombi}
E.~Bombieri and E.~Giusti.
\newblock Local estimates for the gradient of non-parametric surfaces of
  prescribed mean curvature.
\newblock {\em Comm. Pure Appl. Math.}, 26:381--394, 1973.

\bibitem{bucval}
Claudia Bucur, Luca Lombardini, and Enrico Valdinoci.
\newblock Complete stickiness of nonlocal minimal surfaces for small values of
  the fractional parameter.
\newblock {\em Ann. Inst. H. Poincar\'{e} Anal. Non Lin\'{e}aire},
  36(3):655--703, 2019.

\bibitem{nonlocal}
Claudia Bucur and Enrico Valdinoci.
\newblock Nonlocal diffusion and applications.
\newblock {\em Lecture Notes of the Unione Matematica Italiana}, 20:xii+155,
  2016.

\bibitem{CC10}
Xavier Cabr{\'e} and Eleonora Cinti.
\newblock Energy estimates and 1-{D} symmetry for nonlinear equations involving
  the half-{L}aplacian.
\newblock {\em Discrete Contin. Dyn. Syst.}, 28(3):1179--1206, 2010.

\bibitem{CC14}
Xavier Cabr\'e and Eleonora Cinti.
\newblock Sharp energy estimates for nonlinear fractional diffusion equations.
\newblock {\em Calc. Var. Partial Differential Equations}, 49(1-2):233--269,
  2014.

\bibitem{CS15}
Xavier Cabr{\'e} and Yannick Sire.
\newblock Nonlinear equations for fractional {L}aplacians {II}: {E}xistence,
  uniqueness, and qualitative properties of solutions.
\newblock {\em Trans. Amer. Math. Soc.}, 367(2):911--941, 2015.

\bibitem{CM05}
Xavier Cabr{\'e} and Joan Sol{\`a}-Morales.
\newblock Layer solutions in a half-space for boundary reactions.
\newblock {\em Comm. Pure Appl. Math.}, 58(12):1678--1732, 2005.

\bibitem{cafgar}
Luis Caffarelli, Nicola Garofalo, and Fausto Seg\`ala.
\newblock A gradient bound for entire solutions of quasi-linear equations and
  its consequences.
\newblock {\em Comm. Pure Appl. Math.}, 47(11):1457--1473, 1994.

\bibitem{nms}
Luis Caffarelli, Jean-Michel Roquejoffre, and Ovidiu Savin.
\newblock Nonlocal minimal surfaces.
\newblock {\em Comm. Pure Appl. Math.}, 63(9):1111--1144, 2010.

\bibitem{caffy}
Luis Caffarelli and Enrico Valdinoci.
\newblock Regularity properties of nonlocal minimal surfaces via limiting
  arguments.
\newblock {\em Adv. Math.}, 248:843--871, 2013.

\bibitem{CintiSurvey}
Eleonora Cinti.
\newblock Flatness results for nonlocal phase transitions.
\newblock In {\em Contemporary research in elliptic {PDE}s and related topics},
  volume~33 of {\em Springer INdAM Ser.}, pages 247--275. Springer, Cham, 2019.

\bibitem{cozzidg}
Matteo Cozzi.
\newblock Fractional {D}e {G}iorgi classes and applications to nonlocal
  regularity theory.
\newblock In {\em Contemporary research in elliptic {PDE}s and related topics},
  volume~33 of {\em Springer INdAM Ser.}, pages 277--299. Springer, Cham, 2019.

\bibitem{cl}
Matteo Cozzi and Luca Lombardini.
\newblock On nolocal minimal graphs.
\newblock {\em preprint}, 2018.

\bibitem{cozzpass}
Matteo Cozzi and Tommaso Passalacqua.
\newblock One-dimensional solutions of non-local {A}llen-{C}ahn-type equations
  with rough kernels.
\newblock {\em J. Differential Equations}, 260(8):6638--6696, 2016.

\bibitem{PKW08}
Manuel del Pino, Micha{\l} Kowalczyk, and Juncheng Wei.
\newblock A counterexample to a conjecture by {D}e {G}iorgi in large
  dimensions.
\newblock {\em C. R. Math. Acad. Sci. Paris}, 346(23-24):1261--1266, 2008.

\bibitem{nhk}
Agnese Di~Castro, Tuomo Kuusi, and Giampiero Palatucci.
\newblock Nonlocal {H}arnack inequalities.
\newblock {\em J. Funct. Anal.}, 267(6):1807--1836, 2014.

\bibitem{hitch}
Eleonora Di~Nezza, Giampiero Palatucci, and Enrico Valdinoci.
\newblock Hitchhiker's guide to the fractional {S}obolev spaces.
\newblock {\em Bull. Sci. Math.}, 136(5):521--573, 2012.

\bibitem{dipierro2017three}
Serena Dipierro, Alberto Farina, and Enrico Valdinoci.
\newblock A three-dimensional symmetry result for a phase transition equation
  in the genuinely nonlocal regime.
\newblock {\em Calculus of Variations and Partial Differential Equations},
  57(1):57:15, 2018.

\bibitem{improvement}
Serena Dipierro, Joaquim Serra, and Enrico Valdinoci.
\newblock Improvement of flatness for nonlocal phase transitions.
\newblock {\em arXiv preprint arXiv:1611.10105}, 2016.

\bibitem{dipiersurv}
Serena Dipierro and Enrico Valdinoci.
\newblock Nonlocal minimal surfaces: interior regularity, quantitative
  estimates and boundary stickiness.
\newblock pages 165--209, 2018.

\bibitem{dipierro2017decay}
Serena Dipierro, Enrico Valdinoci, and Vincenzo Vespri.
\newblock Decay estimates for evolutionary equations with fractional
  time-diffusion.
\newblock {\em J. Evol. Equ.}, 19(2):435--462, 2019.

\bibitem{poing}
Irene Drelichman and Ricardo~G. Dur\'{a}n.
\newblock Improved {P}oincar\'{e} inequalities in fractional {S}obolev spaces.
\newblock {\em Ann. Acad. Sci. Fenn. Math.}, 43(2):885--903, 2018.

\bibitem{DydaHardy}
Bart\l~omiej Dyda.
\newblock A fractional order {H}ardy inequality.
\newblock {\em Illinois J. Math.}, 48(2):575--588, 2004.

\bibitem{figalli}
Alessio Figalli and Joaquim Serra.
\newblock On stable solutions for boundary reactions: a {D}e {G}iorgi-type
  result in dimension 4+1.
\newblock {\em arXiv preprint arXiv:1705.02781}, 2017.

\bibitem{GGUI}
Nassif Ghoussoub and Changfeng Gui.
\newblock On a conjecture of {D}e {G}iorgi and some related problems.
\newblock {\em Math. Ann.}, 311(3):481--491, 1998.

\bibitem{sunra}
Nicola Gigli and Sunra Mosconi.
\newblock The abstract {L}ewy-{S}tampacchia inequality and applications.
\newblock {\em J. Math. Pures Appl. (9)}, 104(2):258--275, 2015.

\bibitem{trudy}
David Gilbarg and Neil~S. Trudinger.
\newblock {\em Elliptic partial differential equations of second order}.
\newblock Classics in Mathematics. Springer-Verlag, Berlin, 2001.
\newblock Reprint of the 1998 edition.

\bibitem{plapl}
Antonio Iannizzotto, Sunra Mosconi, and Marco Squassina.
\newblock Global {H}\"older regularity for the fractional {$p$}-{L}aplacian.
\newblock {\em Rev. Mat. Iberoam.}, 32(4):1353--1392, 2016.

\bibitem{Korvenpa}
Janne Korvenp{\"a}{\"a}, Tuomo Kuusi, and Erik Lindgren.
\newblock Equivalence of solutions to fractional p-{L}aplace type equations.
\newblock {\em Journal de Math{\'e}matiques Pures et Appliqu{\'e}es}, 2017.

\bibitem{korvenpaaobstacle}
Janne Korvenp\"{a}\"{a}, Tuomo Kuusi, and Giampiero Palatucci.
\newblock The obstacle problem for nonlinear integro-differential operators.
\newblock {\em Calculus of Variations and Partial Differential Equations},
  55(3):63, 2016.

\bibitem{tesiluca}
Luca Lombardini.
\newblock Minimization problems involving nonlocal functionals: nonlocal
  minimal surfaces and a free boundary problem (phd thesis).
\newblock {\em https://arxiv.org/abs/1811.09746}, 2018.

\bibitem{ros2015entire}
Xavier Ros-Oton and Yannick Sire.
\newblock Entire solutions to semilinear nonlocal equations in $\mathbb{R}^2$.
\newblock {\em arXiv preprint arXiv:1505.06919}, 2015.

\bibitem{savin12}
O.~Savin.
\newblock Rigidity of minimizers in nonlocal phase transitions {II}.
\newblock {\em Anal. Theory Appl.}, 35(1):1--27, 2019.

\bibitem{flatty}
Ovidiu Savin.
\newblock Regularity of flat level sets in phase transitions.
\newblock {\em Ann. of Math. (2)}, 169(1):41--78, 2009.

\bibitem{SavinSurvey}
Ovidiu Savin.
\newblock Phase transitions, minimal surfaces and a conjecture of {D}e
  {G}iorgi.
\newblock In {\em Current developments in mathematics, 2009}, pages 59--113.
  Int. Press, Somerville, MA, 2010.

\bibitem{rigidity}
Ovidiu Savin.
\newblock Rigidity of minimizers in nonlocal phase transitions.
\newblock {\em Anal. PDE}, 11(8):1881--1900, 2018.

\bibitem{monoton}
Ovidiu Savin and Enrico Valdinoci.
\newblock Some monotonicity results for minimizers in the calculus of
  variations.
\newblock {\em J. Funct. Anal.}, 264(10):2469--2496, 2013.

\bibitem{densityEs}
Ovidiu Savin and Enrico Valdinoci.
\newblock Density estimates for a variational model driven by the {G}agliardo
  norm.
\newblock {\em J. Math. Pures Appl. (9)}, 101(1):1--26, 2014.

\bibitem{silvestre}
Luis Silvestre.
\newblock Regularity of the obstacle problem for a fractional power of the
  {L}aplace operator.
\newblock {\em Comm. Pure Appl. Math.}, 60(1):67--112, 2007.

\bibitem{SV09}
Yannick Sire and Enrico Valdinoci.
\newblock Fractional {L}aplacian phase transitions and boundary reactions: a
  geometric inequality and a symmetry result.
\newblock {\em J. Funct. Anal.}, 256(6):1842--1864, 2009.

\bibitem{wang}
Xu-Jia Wang.
\newblock Interior gradient estimates for mean curvature equations.
\newblock {\em Math. Z.}, 228(1):73--81, 1998.

\end{thebibliography}
\bibliographystyle{plain}
\end{document}